\documentclass[11pt]{article}

\usepackage[margin=1.25in]{geometry}

\usepackage{amssymb,amsmath}
\usepackage{amsthm}
\usepackage{hyperref}
\usepackage{mathrsfs}
\usepackage{textcomp}
\hypersetup{
    colorlinks,%
    citecolor=black,%
    filecolor=black,%
    linkcolor=black,%
    urlcolor=black
}

\usepackage{verbatim}

\usepackage{sectsty}

\makeatletter

\newdimen\bibspace
\renewenvironment{thebibliography}[1]{%
 \section*{\refname 
       \@mkboth{\MakeUppercase\refname}{\MakeUppercase\refname}}%
     \list{\@biblabel{\@arabic\c@enumiv}}%
          {\settowidth\labelwidth{\@biblabel{#1}}%
           \leftmargin\labelwidth
           \advance\leftmargin\labelsep
           \itemsep\bibspace
           \parsep\z@skip     %
           \@openbib@code
           \usecounter{enumiv}%
           \let\p@enumiv\@empty
           \renewcommand\theenumiv{\@arabic\c@enumiv}}%
     \sloppy\clubpenalty4000\widowpenalty4000%
     \sfcode`\.\@m}
    {\def\@noitemerr
      {\@latex@warning{Empty `thebibliography' environment}}%
     \endlist}

\makeatother

\makeatletter

\newtheorem{thm}{Theorem}[section]
\newtheorem{lem}[thm]{Lemma}
\newtheorem{prop}[thm]{Proposition}

\newtheorem{cor}[thm]{Corollary}

\def\Xint#1{\mathchoice
  {\XXint\displaystyle\textstyle{#1}}%
  {\XXint\textstyle\scriptstyle{#1}}%
  {\XXint\scriptstyle\scriptscriptstyle{#1}}%
  {\XXint\scriptscriptstyle\scriptscriptstyle{#1}}%
  \!\int}
\def\XXint#1#2#3{{\setbox0=\hbox{$#1{#2#3}{\int}$}
  \vcenter{\hbox{$#2#3$}}\kern-.5\wd0}}

\def\dashint{\Xint-}

\newcommand{\al}{\alpha}                \newcommand{\lda}{\lambda}
\newcommand{\om}{\Omega}                \newcommand{\pa}{\partial}
\newcommand{\va}{\varepsilon}           \newcommand{\ud}{\mathrm{d}}
\newcommand{\be}{\begin{equation}}      \newcommand{\ee}{\end{equation}}
                 
\newcommand{\Lda}{\Lambda}              \newcommand{\B}{\mathcal{B}}
\newcommand{\R}{\mathbb{R}}

\begin{document}

\title{\textbf{Isolated singularities of solutions to the  Yamabe equation in dimension $6$}
\bigskip}

\author{\medskip
Jingang Xiong\footnote{J. Xiong is partially supported by NSFC 11922104 and 11631002.},\  \ \
Lei Zhang \footnote{ L. Zhang is partially supported by a collaboration grant of Simons Foundation.}}

\date{\today}

\maketitle

\begin{abstract} We study the asymptotic behavior of local solutions to the Yamabe equation near an isolated singularity, when the metric is not conformally flat. We prove that, in dimension $6$, any solution is asymptotically close to a Fowler solution, which is an extension of the same result for lower dimensions by F.C. Marques in 2008.

\end{abstract}

\section{Introduction}

Let $g$ be a smooth Riemannian metric on the  unit ball $B_1\subset \R^n$, $n\ge 3$. We study positive solutions of the Yamabe equation in the punctured ball
\begin{equation}\label{Y0}
-L_g u=n(n-2)u^{\frac{n+2}{n-2}}\quad \mbox{in }B_1\setminus \{0\},
\end{equation}
where
$L_g = \Delta_g- \frac { (n-2) }{ 4(n-1) }R_g $ is the conformal Laplacian,
 $\Delta_g$ is the Laplace-Beltrami operator and $R_g$ is the
scalar curvature of $g$. We will always assume that solutions are smooth away from the singular point.

 When $g$ is conformally flat  and $0$ is a non-removable singularity, Caffarelli-Gidas-Spruck \cite{CGS} proved that
\be \label{eq:CGS}
u(x)=u_0(x)(1+o(1))\quad \mbox{as }x \to 0,
\ee
where  $u_0$ is a Fowler solution. Here Fowler solutions are referred to the singular positive solutions of
\[
-\Delta u_0= n(n-2)u_0^{\frac{n+2}{n-2}} \quad \mbox{in }\R^n\setminus \{0\},
\]
which were proved to be radially symmetric and classified in the same paper \cite{CGS}. A different proof and refinement of this result were given by  Korevaar-Mazzeo-Pacard-Schoen \cite{KMPS}, in particular, they improved the $o(1)$ remainder term to a $O(|x|^{\al})$ for some $\al>0$. Namely,
\[
u(x)=u_0(x)(1+O(|x|^{\alpha})).
\]

When $g$ is not conformally flat and $3\le n\le 5$, Marques \cite{marques-1} established the same asymptotic behavior. In this paper, we show that this still holds in  dimension $6$, which appears to be the borderline of the current method.

\begin{thm}\label{asy-s-1} Suppose that $n=6$ and $u\in C^2(B_1\setminus \{0\})$ is a positive solution of \eqref{Y0}.  If  $0$ is not a removable singularity of $u$,  then
\be  \label{eq:asy-2}
u(x)=u_{0}(1+O(|x|^{\al}))\quad \mbox{as }x\to 0,
\ee
where $u_0$ is a Fowler solution and $\al>0$.
\end{thm}

Once the convergence to a Fowler solution is established,  the arguments of \cite{KMPS} and \cite{marques-1} can be used to improve the approximation by deformed Flower solutions. See Han-Li-Li \cite{HLL} for expansions up to arbitrary orders when the metric is conformally flat. Existence of solutions of \eqref{Y0} is related to the study of local solutions of the singular Yamabe problem, which has been studied by Schoen \cite{S1}, Mazzeo-Smale \cite{MS}, Mazzeo-Pollack-Uhlenbeck \cite{MPU}, Mazzeo-Pacard \cite{MP} and etc, after the resolution of the Yamabe problem by Yamabe \cite{Y}, Trudinger \cite{T}, Aubin \cite{A} and Schoen \cite{S}.

The difficulty to establish asymptotical symmetry of solutions near isolated singularities is that \eqref{Y0} has no symmetry when  $g$ is not conformally flat.  Similar difficulty also happens to the prescribing scalar curvature equation
\be \label{eq:chen-lin}
-\Delta u=n(n-2) K u^{\frac{n+2}{n-2}} \quad \mbox{in }B_1\setminus \{0\}, \quad u>0,
\ee
where $\Delta $ is the Laplace operator and $K>0$ is a $C^1$ function in $B_1$. Under the flatness condition
\be \label{eq:chen-lin-1}
c_1 |x|^{l-1}\le |\nabla K(x)| \le c_1|x|^{l-1},
\ee
where $c_1, c_2>0$ and $l\ge \frac{n-2}{2}$ are constants,  Chen-Lin \cite{ChenLin2, ChenLin3}  established \eqref{eq:CGS} for non-removable singularities.
On the other hand, they constructed a singular solution which does not satisfy \eqref{eq:CGS}  when $l<\frac{n-2}{2}$. As for \eqref{Y0}, the metric in normal coordinates centered at any point $\bar x\in B_1$ has the flatness
\[
g_{ij}=\delta_{ij}+O(|x|^\tau)
\]
with $\tau=2$. Marques' proof would be possible to give \eqref{eq:asy-2} in all dimensions, if  $\tau>\frac{n-2}{2}$ in a neighborhood of $0$. Obviously,  this condition holds automatically in dimensions in $3,4,5$ but does not in dimension $6$.

By Theorem 8 of \cite{marques-1}, \eqref{eq:asy-2} follows from 
\begin{equation}\label{b-6}
\frac{1}{C}d_g(x,0)^{-\frac{n-2}{2}}\le u(x)\le C d_g(x,0)^{-\frac{n-2}{2}}
\end{equation}
for some $C\ge 1$ independent of $x$, where $d_g$ is the distance function with respect to $g$. The proof of both the upper bound and lower bound in (\ref{b-6}) for dimension $6$ requires delicate analysis to handle difficulties related to borderline cases. To establish the upper bound we use the moving spheres method which requires a careful construction of a test function. In order to overcome certain difficulties we build our argument on properties of the conformal normal coordinates and apply the maximum principle only on selected domains.  See  \cite{LZhu, LZhang, zhang-li-low-dim, LZ2, Z} about the moving spheres method. To prove the lower bound in dimension $6$, we deform the metric conformally to one with negative scalar curvature and take advantage of certain monotonicity properties of the solutions. As a result we can  improve a differential inequality of \cite{marques-1} that plays a crucial role in the proof of the lower bound. The dimension $6$ case shows some similarity to the borderline $l=\frac{n-2}{2}$ of  \cite{ChenLin2, ChenLin3}, but the proof in this article is more involved.

To end this section, we would like to mention some related papers about the isolated singularities problem for the Yamabe equation; see \cite{caju, CP-1, CP-2, HLT, JX, Lc, Li,  SS, Ta-Zhang, Z} and references therein.

This paper is organized as follows. In section \ref{s:2},  we establish the upper bound.
In section \ref{s:3}, we establish the criteria of singularity removability in terms of Pohozaev integral and prove the main theorem.

\bigskip

\noindent \textbf{Acknowledgement:} This work was completed while J. Xiong was visiting Rutgers University, to which he is grateful for providing  the very stimulating research environments and supports. Both authors would like to thank Professor YanYan Li for his insightful guidance and constant encouragement.

\section{The upper bound}
\label{s:2}

We will use $\B_\rho^g(x)$ to denote the geodesic ball with respect to $g$ centered at $x$ with radius $\rho>0$, the superscript $g$ in $\B_\rho^g(x)$ will be dropped when there is no ambiguity.

\begin{thm}\label{sphe-har} Suppose $n=6$ and $u\in C^2(B_1 \setminus \{0\})$ is a solution of \eqref{Y0}.  Then
\begin{equation}\label{sphe-har-e}
\limsup_{x\to 0} d_g(x,0)^{\frac{n-2}2} u(x)<\infty.
\end{equation}
\end{thm}

\begin{proof}

Without loss of generality, we assume $B_1$ is the normal coordinates chart of $g$ centered at $0$.  If \eqref{sphe-har-e} were invalid,  there exists a sequence $x_k\to 0$ such that
\be\label{eq:contra-hy-1}
d_g(0,x_k)^{\frac{n-2}2} u(x_k)\to \infty \quad \mbox{as }k\to \infty.
\ee
We shall divide the remaining proof into four steps.

\medskip

\textbf{Step 1. Blow-up analysis.}

\medskip

\textbf{Claim:} The sequence  $x_k$ in \eqref{eq:contra-hy-1} can be selected to be local maximum points of $u$.

The proof of this fact is standard and we briefly describe it for readers' convenience. Set
$$f_k(y)=u(y)(d_k-d_g(y,x_k))^{\frac{n-2}2}\quad \mbox{for } d_g(y,x_k)\le d_k, $$
where $d_k=d_g(x_k,0)/2$.
Clearly, $f_k(x_k)\to \infty$ and $f_k=0$ on $\partial \B_d(x_k)$. Let $f_k(\hat x_k)$ be a maximum of $f_k$ on $\B_{d_k}(x_k)$ and set
$$\alpha_k=\frac 12(d_k-d_g(\hat x_k,x_k)).$$ By the definition of $\hat x_k$ we have
\be \label{eq:scale_to_Rn}
u(\hat x_k)(2\alpha_k )^{\frac{n-2}2}\ge u(x_k)d_k^{\frac{n-2}2}\to  \infty
\ee
and for $y\in \B_{\alpha_k}(\hat x_k)$,
\begin{equation}\label{u-up}
u(y)\le u(\hat x_k)(\frac{2\alpha_k}{d_k-d_g(y,x_k)})^{\frac{n-2}2}\le  u(\hat x_k) 2^{\frac{n-2}2},
\end{equation}
where we have used  $d_k-d_g(y,x_k)\ge d_k-d_g(\hat x_k,x_k)-d_g(\hat x_k,y)\ge \alpha_k $ in the last inequality.

As a consequence of \eqref{eq:scale_to_Rn}, \eqref{u-up}, the sequence of functions $\hat v_k$ defined by
$$\hat v_k(y)=u(\hat x_k)^{-1}u(\exp_{\hat x_k}u(\hat x_k)^{-\frac{2}{n-2}}y)$$
 has a subsequence (still denoted as $v_k$) that converges in $C^2_{loc}(\mathbb R^n)$ to $U$ of
\begin{equation}\label{eq:limit-bubble}
\Delta U+n(n-2)U^{\frac{n+2}{n-2}}=0\quad \mbox{in } \mathbb R^n.
\end{equation}
By the classification theorem of Caffarelli-Gidas-Spruck \cite{CGS},
\[
U(y)= \Big(\frac{\lda}{1+\lda^2 |y-y_0|^2}\Big)^{\frac{n-2}{2}}
\]
for some $\lda\ge 1$ and $y_0\in \R^n$.  Since $y_0$ is the maximal point of $U$ and $\nabla^2 U(y_0)$ is negative definite, there exist $y_k\to y_0$ where $y_k$ is a local maximum of $v_k$. Thus from the beginning we can assume $x_k$ to be the pre-image of $y_k$. Thus $x_k$ is a local maximum point of $u$ which also satisfies
\be \label{eq:scale_to_Rn-1}
u(x_k)\alpha_k ^{\frac{n-2}2}\to  \infty \quad \mbox{as }k\to \infty.
\ee

We shall use the conformal normal coordinates centered at $x_k$.
Namely, we can  find a smooth positive function $\kappa_k$ to deform the metric conformally: $\bar g:= \kappa_k^{-\frac{4}{n-2}} g$. In
the normal coordinates  of $\bar g$ centered at $x_k$ there holds
\[
\det \bar g(x)= 1 \quad \mbox{for }|x|<\delta,
\]
where $\delta >0$ is independent of $k$; see Cao \cite{Cao} and G\"unther \cite{G}.  (Without causing much confusion, we did not label $k$ to $\bar g$.) 
It is easy to check that
\be \label{eq:conformal-normal}
\kappa_k(0)=1 \quad \mbox{and} \quad \nabla \kappa_k(0)=0.
\ee
Let $u_k=\kappa_k u$, then $u_k$ satisfies the following equation based on the conformal invariant property of $L_g$:
\[
-L_{\bar g} u_k(x)= n(n-2) u_k(x)^{\frac{n+2}{n-2}} \quad \mbox{in }B_{\delta }\setminus \{z_k\},
\]
where $z_k$ is the singular point in the new coordinate and $\delta$ is a positive small number. For scaling we set
 $M_k= u_k(x_{k}) $ and
\[
v_k(y)= M_k^{-1} u_k(\exp_{x_k} (M_k^{-\frac{2}{n-2}} y)).
\]
Then the conformal invariant property further carries us to the equation for $v_k$:
\be \label{eq:rescaled}
-L_{ g_k} v_k(y)= n(n-2)  v_k(y)^{\frac{n+2}{n-2}} \quad \mbox{in }B_{\delta  M_k^{\frac{2}{n-2}} }\setminus \{ S_k\},
\ee
where $(g_k)_{ij}(y)=\bar g_{ij}(M_k^{-\frac{2}{n-2}}  y)$ and $ S_k=M_k^{-\frac{2}{n-2}}  z_k $.
By the discussion about the location of $x_k$ we have
\[
|S_k | \to \infty
\]
 and
\[
v_k(y) \to U(y) \quad \mbox{in }C_{loc}^2 (\R^n)
\]
as $k\to \infty$, where  $U\ge 0$ satisfies \eqref{eq:limit-bubble}.
Since $x_k$ is a local maximum of $u$ and \eqref{eq:conformal-normal} holds, we have
\[
U(0)=1\quad \mbox{and} \quad \nabla U(0)=0.
\]
By the classification theorem of Caffarelli-Gidas-Spruck \cite{CGS},
\[
U(y)= \left(\frac{1}{1+|y|^2}\right)^{\frac{n-2}{2}}.
\]

\medskip

Before further investigation we mention two lower bounds of $v_k$ which will be used in two different occasions. The first one is
\begin{equation}\label{vk-lb-2}
v_k(y)\ge \Lambda M_k^{-1},\quad |y|\le \delta M_k^{\frac{2}{n-2}},
\end{equation}
which follows from the maximum principle and the definition of $v_k$. Indeed, choose $\delta>0$ small so that the first eigenvalue of $-L_g$ is positive in $B_{\delta}$ (with respect to Dirichlet boundary condition). Let $u\ge \Lambda_1>0$ on $\partial B_{\delta}$ and $\phi$ be the solution of
$L_g \phi=0$ in $B_{\delta}$ with $\phi=\Lambda_1$ on $\partial B_{\delta}$. Then we see that $u\ge \phi$ by the maximum principle and $\phi>\Lambda>0$ for some $\Lambda\in (0,\Lambda_1)$ by standard Harnack inequality. Thus (\ref{vk-lb-2}) holds. The second lower bound is stated in the following proposition:

\begin{prop}\label{prop:l-b-1} There exists $C>0$ independent of $k$ such that
\be \label{low-b-v}
v_k(y) \ge \frac{1}{C} (1+|y|)^{2-n} \quad \mbox{for }y\in  B_{\delta M_k^{\frac{2}{n-2}}}.
\ee
\end{prop}

\begin{proof}  Fix $\tau>0$ so that $-L_g$ is coercive in $H_0^1(B_{\tau})$.  We assume $\delta<\tau/2$.
 Let $G_k$ be the solution of
\[
-L_{g} G_k(x)= \delta_{x_k} \quad \mbox{in }B_{\tau}, \quad G_k(x)=0 \quad \mbox{on }\pa B_{\tau},
\]
where $\delta_{x_k}$ is the Dirac measure centered at $x_k$. Then $G_k$ satisfies
\begin{equation} \label{6new}
\begin{split}
&\frac 1A |y|^{2-n}\le G_k(\exp_{x_k}y)\le A
|y|^{2-n} \quad \mbox{for }
y\in B_{ \delta}\setminus\{0\}, \\
&\lim_{y\to 0} G(\exp_{x_k}y)|y|^{n-2}=\frac 1{  (n-2) \omega_n},
\end{split}
\end{equation}
where $\omega_n$ denotes the volume of the
standard $(n-1)$-sphere and $A>0$ is independent of $k$.
Since $v_k(y)\to U(y)$ as $k\to \infty$ for $|y|=1$,  there exists $C>0$ independent of $k$ such that
\begin{align*}
u \ge \frac{1}{C} G_k  \quad \mbox{on } \pa \B_{M_k^{-\frac{2}{n-2}}}(x_k).
\end{align*}
By the comparison principle, we have $u\ge \frac{1}{C} G_k$ in $B_{\tau}\setminus \B_{M_k^{-\frac{2}{n-2}}}(x_k)$. Hence,
\begin{align*}
v_k(y) &\ge \frac{1}{C} M_k^{-1} G_k(\exp_{x_k} M_k^{-\frac{2}{n-2}}y)\\&
 \ge \frac{1}{AC} M_k^{-1} M_k |y|^{2-n} =\frac{1}{AC} |y|^{2-n} \quad \mbox{for }|y|\ge 1.
\end{align*}
When $|y|\le 1$, we used $v_k\to U$ again to have the lower bound. Therefore, the proposition is proved.
\end{proof}

Recall that in the  conformal normal coordinates,
$$
\bar g_{ij}(x)=\delta_{ij}+O(|x|^2), \quad \det \bar g=1, \quad   R_{\bar g}(x)=O(|x|^2).$$
It follows that
$$\Delta_{\bar g}=\partial_i(\bar g^{ij}\partial_j)=\Delta+b_j\partial_j+d_{ij}\partial_{ij}, $$
where
\[
b_j(x)=\pa_i \bar g^{ij}(x)=O(|x|),
\] and
\[
d_{ij}(x)=\bar g^{ij}(x)-\delta_{ij}=O(|x|^2).
\]
Thus
\[ \label{v-eq}
-L_{g_k} =-\Delta_{g_k}  +c(n)R_{g_k}  = -\Delta  -\bar  b_j\partial_j -\bar d_{ij}\partial_{ij} +\bar c ,
\]
where $c(n)=\frac { (n-2) }{ 4(n-1) }$,
\be \label{rough}
\begin{split}
& \bar b_j(y)=M_k^{-\frac{2}{n-2}}b_j(M_k^{-\frac{2}{n-2}}y)=O(M_k^{-\frac{4}{n-2}})|y|,\\
& \bar d_{ij}(y)=d_{ij}(M_k^{-\frac{2}{n-2}}y)=O(M_k^{-\frac{4}{n-2}})|y|^2,\\
&\bar  c(y)= c(n)R_{\bar g}(M_k^{-\frac{2}{n-2}}y)M_k^{-\frac{4}{n-2}}= O(M_k^{-\frac{8}{n-2}})|y|^2.
\end{split}
\ee
Note that the subscripts $k$ are dropped for convenience.
The equation of $v_k$ becomes
\begin{equation}\label{vk-e2}
(\Delta+\bar b_j\partial_j+\bar d_{ij}\partial_{ij}-\bar c)v_k+n(n-2)v_k^{\frac{n+2}{n-2}}=0\quad \mbox{in }B_{\delta M_k^{\frac 2{n-2}}}\setminus \{S_k\}.
\end{equation}

\medskip

\textbf{Step 2. Setting up the moving spheres framework.}

\medskip

For  $\lambda>0$ and any function $v$, define
$$v^{\lambda}(y):=\Big(\frac{\lambda}{|y|}\Big)^{n-2}
v(y^\lambda),  \qquad y^\lambda:=\frac{\lambda^2y}{|y|^2} $$
as the Kelvin transformation of $v$ with respect to $\partial B_{\lambda}$. To carry out the method of moving spheres we restrict our discussion on $\Sigma_{\lambda}\setminus \{S_k\}$,  where $\Sigma_{\lambda}$ is defined as
$$
\Sigma_{\lambda}:=B_{\delta M_k^{\frac{2}{n-2}}} \setminus
\bar B_ \lambda =\{y\ |\ \lambda<|y|<
 \delta  M_k^{\frac{2}{n-2}}\}.
$$
Let
$$
w_{\lambda}(y):= v_k(y)-v_k^{\lambda}(y),\qquad y\in \Sigma_{\lambda}\setminus \{S_k\}.
$$
A straight forward computation yields
\begin{equation}
\Delta w_{\lambda}+\bar b_i\partial_i w_{\lambda}
+\bar d_{ij}\partial_{ij}w_{\lambda}
-\bar cw_{\lambda}+n(n+2)\xi^{\frac 4{n-2}}w_{\lambda}=
E_{\lambda} \quad\mbox{in } \Sigma_\lambda\setminus \{S_k\},
\label{3eq5}
\end{equation}
where $\xi>0$ is given by
\begin{equation}
n(n+2)\xi^{ \frac 4{n-2}}=\left\{\begin{array}{ll}
n(n-2) \frac { v_k^{ \frac{n+2}{n-2} }-(v_k^\lambda)^{ \frac{n+2}{n-2} } }
{ v_k-v_k^\lambda },\quad v_k\neq v_k^{\lambda},\\
\\
n(n+2)v_k^{\frac{4}{n-2}},\quad v_k=v_k^{\lambda},
\end{array}
\right.
\label{stays}
\end{equation}
and
\begin{align}
E_{\lambda}(y)=& \left(\bar c(y)v_k^{\lambda}(y)- (\frac{\lambda}{|y|})^{n+2}
\bar c(y^{\lambda})v_k(y^{\lambda})\right)
-(\bar b_j\partial_j v_k^{\lambda}+\bar d_{ij}\partial_{ij}v^{\lambda}_k)
\nonumber\\
& +(\frac{\lambda}{|y|})^{n+2}\left(\bar b_j(y^{\lambda})
\partial_j v_k(y^{\lambda})+
\bar d_{ij}(y^{\lambda})\partial_{ij}v_k(y^{\lambda})\right).
\label{eQ}
\end{align}
Here we note that we shall always require $\lambda\in [1/2,2]$. Since $|S_k| \to \infty$ as $k\to \infty$,  $v_k$ is smooth in $B_{\lda}$ and $v_k^{\lambda}$ is smooth in $\Sigma_{\lambda}$.

The following estimate of $E_{\lambda}$ is crucial to the construction of auxiliary functions in the sequel. Since it is related to the smallness of $v_k-U$ in $B_2$, we use the a notation to represent this quantity:
\be
\sigma_k:=
\|v_k-U\|_{C^2(B_2)}, \quad \sigma_k\to 0 \quad \mbox{as }k\to \infty. 
\ee

\begin{prop}\label{prop:lz2} Let $E_{\lambda}$ be defined in (\ref{eQ}), then for
 $\lambda\in [1/2,2]$ and $y\in \Sigma_\lambda$, we have
            \begin{equation} \label{eq:E-lambda}
|E_\lambda| \le
 C_0 M_k^{-\frac{8}{n-2}}|y|^{4-n}
+C_0 \sigma_k
M_k^{ -\frac 4{n-2} } |y|^{-n},
\end{equation}
where  $C_0>0$ is some constant independent of $y$ and $k$.
\label{prop2new}
\end{prop}

Proposition \ref{prop:lz2} is an easy corollary of  Proposition 2.1 of \cite{LZ2}. We include a proof here for readers' convenience.

\begin{proof}[Proof of Proposition \ref{prop:lz2}]
First we estimate the second term of $E_{\lambda}$:
$$I:=(\bar b_j\partial_j v_k^{\lambda}+\bar d_{ij}\partial_{ij}v^{\lambda}_k).$$
Since in  the conformal normal coordinates
\begin{equation}
0=(\Delta_{g_k}-\Delta)w=
(\bar b_j\partial _j+\bar d_{ij}\partial_{ij})w
\label{rrr}
\end{equation}
for any smooth radial function $w(y)$,
we have
$$
I=(\bar b_j\partial_j+\bar d_{ij}\partial_{ij})
[ (v_k-U)^\lambda].
$$
By a direct computation,
$$
\partial_j
\bigg\{  (\frac \lambda{|y|})^{n-2}  (v_k-U)(y^\lambda)\bigg\}
=
\partial_j
\bigg\{  (\frac \lambda{|y|})^{n-2} \bigg\}  (v_k-U)(y^\lambda)
+
  (\frac \lambda{|y|})^{n-2}
\partial_j \bigg\{   (v_k-U)(y^\lambda)
\bigg\},
$$
\begin{align*}
&\partial_{ij}\bigg\{  (\frac \lambda{|y|})^{n-2}  (v_k-U)(y^\lambda)\bigg\}
\\&=\partial_{ij}\bigg\{  (\frac \lambda{|y|})^{n-2} \bigg\}  (v_k-U)(y^\lambda)+
\partial_i \bigg\{  (\frac \lambda{|y|})^{n-2} \bigg\}
\partial_j  \bigg\{    (v_k-U)(y^\lambda)  \bigg\}\\ & \quad
+ \partial_j
\bigg\{  (\frac \lambda{|y|})^{n-2} \bigg\}
\partial_i
 \bigg\{    (v_k-U)(y^\lambda)  \bigg\}
+
(\frac \lambda{|y|})^{n-2}
\partial_{ij}
 \bigg\{    (v_k-U)(y^\lambda)  \bigg\}.
\end{align*}
Since  $\bar d_{ij}\equiv \bar d_{ji}$, using \eqref{rrr} with $w=(\frac \lambda{|y|})^{n-2}$ we have
\begin{align*}
I=&
(\frac \lambda{|y|})^{n-2}
\bar b_j \partial_j
 \bigg\{    (v_k-U)(y^\lambda)  \bigg\}
+2\bar d_{ij} \partial_i
\bigg\{  (\frac \lambda{|y|})^{n-2} \bigg\}
\partial_j
 \bigg\{    (v_k-U)(y^\lambda)  \bigg\}
\\&+(\frac \lambda{|y|})^{n-2}
\bar d_{ij} \partial_{ij}
 \bigg\{    (v_k-U)(y^\lambda)  \bigg\}.
\end{align*}
To evaluate terms in $I$, we observe that for $z\in B_2$,
\begin{equation} \label{13-1}
\begin{split}
(v_k-U)(z)&=O(\sigma_k)|z|^2, \\
|\nabla_z (v_k-U)(z)|&=O(\sigma_k)|z|,  \\ |\nabla_z^2(v_k-U)(z)|&=O(\sigma_k),
\end{split}
\end{equation}
where we have used  $(v_k-U)(0)=|\nabla (v_k-U)(0)|=0$.
It follows from \eqref{rough} and (\ref{13-1}) that
\be \label{eq:error-1}
\begin{split}
I&= O(1) \sigma_k  M_k^{ -\frac 4{n-2} } \Big(|y|^{2-n}  |y| |y^\lda | |\nabla_y y^\lda| +|y|^2|y|^{1-n} |y^{\lda}|  |\nabla_y y^\lda| \\& \quad +|y|^{2-n} |y|^2  (|y^\lda | |\nabla_y^2 y^\lda| +|\nabla_y y^\lda| ^2 ) \Big) \\& = O(1)\sigma_k
M_k^{ -\frac 4{n-2} } |y|^{-n}.
\end{split}
\ee
Similarly,
$$
(\frac{\lambda}{|y|})^{n+2}\left(\bar b_j(y^{\lambda})
\partial_j v_k(y^{\lambda})+
\bar d_{ij}(y^{\lambda})\partial_{ij}v_k(y^{\lambda})\right)
= O(1)\sigma_k
M_k^{ -\frac 4{n-2} } |y|^{-n}
$$
and
$$
 |\bar c(y)||v_k^\lambda(y)-U^\lambda(y)|
+(\frac \lambda{|y|})^{n+2}
|\bar c(y^\lambda)||v_k(y^\lambda)-U(y^\lambda)|
=O(1)\sigma_k
M_k^{ -\frac 4{n-2} } |y|^{-n}.
$$
Finally, the estimate on $\bar c$ gives
$$\bar c(y)U^{\lambda}(y)-(\frac{\lambda}{|y|})^{n+2}\bar c(y^{\lambda})U(y^{\lambda})=O(M_k^{-\frac{8}{n-2}})|y|^{4-n}. $$

Therefore, Proposition \ref{prop2new} is established.

\end{proof}

\medskip

\textbf{Step 3. Constructing an auxiliary function. }

\medskip

If $n=6$, \eqref{eq:E-lambda} reads
\begin{equation}\label{e-lam-6}
|E_{\lda}(y)|\le C_0 \sigma_kM_k^{-1}|y|^{-6}+C_0 M_k^{-2}|y|^{-2}\quad \mbox{for }y\in \Sigma_{\lda}.
\end{equation}


Consider the linear ordinary differential equation
\be \label{eq:h-1}h_{\lambda}''(r)+\frac{5}{r}h_{\lambda}'(r)=-2C_0\sigma_kM_k^{-1}r^{-6}-2C_0M_k^{-2}r^{-2}\quad \mbox{for }\lambda<r<\delta M_k^{\frac 12},
\ee
with the initial data
\begin{equation}\label{h-lam-2}
h_{\lambda}(\lambda)=h_{\lambda}'(\lambda)=0.
\end{equation}
It is easy to find out that
\begin{equation} \label{eq:h-lda}
h_{\lambda}(r)=\frac{C_0}{2}\sigma_k M_k^{-1}(r^{-4}\ln \frac r{\lambda}+\frac{r^{-4}}4-\frac{\lambda^{-4}}4)-\frac{C_0}{2}M_k^{-2}(\ln \frac{r}{\lambda} +\frac 14(\frac{\lambda^4}{r^4}-1))
\end{equation}
is the unique solution. An immediate observation is that $h_\lda(r)\le 0$ because $h_\lda$ is super-harmonic, (\ref{h-lam-2}) forces $h_{\lambda}$ to be negative for $r>\lambda$.

Setting $h_{\lambda}(y)=h_{\lambda}(|y|)=h_{\lambda}(r)$, we have
$$\Delta_{g_k}h_{\lambda}=\Delta h_{\lambda}=h_{\lambda}''(r)+\frac 5rh_{\lambda}'(r)=-2C_0\sigma_kM_k^{-2}|y|^{-2}-2C_0M_k^{-1}r^{-6}. $$
By \eqref{e-lam-6},
\be \label{eq:AF-1}
\Delta_{g_k}h_{\lambda}+E_{\lambda} <-C_0 \sigma_kM_k^{-1}|y|^{-6} -C_0 M_k^{-2}|y|^{-2}\quad \mbox{in }\Sigma_{\lambda}.
\ee
Next, we verify that
\begin{equation}\label{h-must}
(\Delta_{g_k}-\bar c+48\xi)h_{\lambda}(y)+E_{\lambda}(y)<0 \quad \mbox{for }y\in \om_{\lda},
\end{equation}
where
\[
 \om_{\lda}:=\left\{y\in  \Sigma_{\lambda}\setminus \{S_k\}|  v_k(y)<2v_k^{\lambda}(y)+2|h_{\lambda}(y)|\right\} .
\]
Indeed, since $|\bar c|\le AM_k^{-2}|y|^2$ for some $A>0$, by the lower bound of $v_k$ in (\ref{low-b-v}),
\be  \label{eq:positive-32}
48\xi-\bar c\ge C|y|^{-4}-AM_k^{-2}|y|^2>0,\quad \mbox{if}\quad |y|<\delta_1 M_k^{1/3}
\ee
for some $\delta_1>0$. On the other hand, for $|y|\in [\delta_1 M_k^{1/3},\delta M_k^{\frac 12})$ and large $k$,
\begin{align} \label{eq:outer}
2v_k^{\lambda}(y)+2|h_{\lambda}(y)|  \le C |y|^{-4} + C\sigma_k M_k^{-1}+ C M_k^{-2} \ln |y| < \Lambda M_k^{-1} \le v_k(y),
\end{align}
where we have used \eqref{vk-lb-2} in the last inequality. Hence,
\be \label{eq:positive-33}
\om_\lda \subset B_{\delta_1 M_k^{1/3}}.
\ee
Since $h_{\lambda}\le 0$,  \eqref{h-must} follows immediately from \eqref{eq:AF-1}, \eqref{eq:positive-32} and \eqref{eq:positive-33}.

\medskip

\textbf{Step 4. Completing the proof of the upper bound of $u$. }

\medskip

The benchmark of the moving sphere method is the following inequality that can be verified by direct computation:
\be \label{eq:f-1}
U(y)-U^\lda(y) > (=, <) 0 \quad \mbox{for }|y|>\lda, \quad \mbox{if }\lda<(=,>)1.
\ee

First we show that
\begin{equation}\label{v-k-s}
w_{\lambda_0}+h_{\lambda_0}>0 \quad \mbox{in}\quad \Sigma_{\lambda_0}\setminus \{S_k\}, \quad \mbox{for}\quad \lambda_0\in [\frac 12,\frac 35].
\end{equation}

The proof of (\ref{v-k-s}) starts from (\ref{eq:f-1}): For $\lambda_0\in [\frac 12,\frac 35]$, there is a universal constant $\epsilon_0>0$ such that
$$U(y)-U^{\lambda_0}(y)>\epsilon_0(|y|-\lambda_0)|y|^{-5}\quad \mbox{for } |y|>\lambda_0. $$
By the convergence of $v_k$ to $U$ in $C^2_{loc}(\R^n)$,  for any fixed $R>>1$,
\be \label{eq:ms-start-1}
v_k(y)-v_k^{\lambda_0}(y)>\frac{\epsilon_0}2(|y|-\lambda_0)|y|^{-5}, \quad \mbox{if } \lambda_0<|y|<R
\ee
and $k$ is sufficiently large. In particular for $|y|=R$,
\be \label{eq:ms-start-2}
v_k(y)\ge (1-\frac{\epsilon_0}2)|y|^{-4} \quad \mbox{and}\quad  v_k^{\lambda_0}(y)\le (1-3\epsilon_0)|y|^{-4}\quad \mbox{for } |y|=R.
\ee
Thus the gap between $v_k$ and $v_k^{\lambda_0}$ is enough to engulf $h_{\lambda_0}$. By
the explicit expression of $h_{\lda}$ and \eqref{eq:ms-start-1}, we see that for large $k$
$$w_{\lambda_0}(y)+h_{\lambda_0}(y)>0\quad \mbox{for } \lambda_0<|y|<R. $$
To prove \eqref{v-k-s} for $R<|y|<\delta M_k^{\frac 12}$, we first determine an upper bound for $\bar c$ and construct a test function $\phi$ over this region.
By \eqref{rough}, we can find $A>0$ to have
\begin{equation}\label{bound-c}
|\bar c|\le AM_k^{-2}|y|^2.
\end{equation}
Then we set
$$\phi(y)=(1-\epsilon_0)|y|^{-4}+\frac{\Lambda}{2M_k}+AM_k^{-2}|y|^2, \quad R<|y|<\delta M_k^{1/2},$$
where $\Lda$ is the constant in \eqref{vk-lb-2}.
It is easy to check that
$$L_{g_k}\phi=\Delta_{g_k}\phi-\bar c\phi=\Delta \phi-\bar c\phi\ge \frac{7A}{M_k^2}-AM_k^{-2}|y|^{-2}-A\Lambda |y|^2M_k^{-3}.$$
By choosing $\delta >0$ small enough (independent of $k$ when $k$ is large), we have
\[
L_{g_k}\phi >0,\quad R<|y|<\delta M_k^{1/2}.
\]
Then the standard maximum principle gives $v_k\ge \phi$ on this annulus because $L_{g_k}(v_k-\phi)<0$ and (see (\ref{eq:ms-start-2}) and (\ref{vk-lb-2}))
\[
v_k >\phi \quad \mbox{on }\partial B_{\delta M_k^{1/2}}\cup \partial B_R.
\]

Since
\begin{equation}\label{ub-vlam}
v_k^{\lambda_0}(y)\le (1-2\epsilon_0)|y|^{-4} \quad \mbox{for } |y|>R
\end{equation}
for large $k$ and $R$, we have
\[
v_k^{\lambda_0}(y)-h_{\lambda_0}\le \phi(y)\quad \mbox{for } |y|>R.
\]
Hence, we conclude that \eqref{v-k-s} holds because
$$v_k(y)-v_k^{\lambda_0}(y)+h_{\lambda_0}(y)>0\quad \mbox{for }R<|y|<\delta M_k^{1/2}. $$

The critical position in the moving sphere method is defined by
$$\bar \lambda:=\sup\{\lambda \in [1/2, 2]|  v_k(y)>v_k^{\mu}(y)-h_{\mu}(y), \quad \forall ~y\in \Sigma_{\mu}\setminus \{S_k\} \mbox{ and } 1/2<\mu<\lambda \}.$$
By \eqref{v-k-s}, $\bar \lda$ is well-defined. In order to reach to the final contradiction we claim that $\bar \lda=2$.

If $\bar \lambda<2$,  by \eqref{eq:outer} we still have $v_k>v_k^{\bar \lambda}
-h_{\bar \lambda}$ on $\partial B_{\delta M_k^{\frac 12}}$. By the maximum principle,  $v_k-v_k^{\bar \lambda}+h_{\bar \lambda}$ is strictly positive in $\Sigma_{\bar \lambda}$ and $\frac{\pa }{\pa r}(v_k-v_k^{\bar \lambda}+h_{\bar \lambda})>0$ on $\partial B_{\bar \lambda}$.  By a standard argument in moving spheres method,  we can move spheres a little further than $\bar \lambda$. This contradicts the definition of $\bar \lambda$. Therefore the claim is proved.

Sending $k$ to $\infty$ in the inequality
\[
v_k(y)>v_k^{\bar \lda }(y)-h_{\bar \lda}(y) \quad \mbox{for }\bar \lda <|y|<\delta M_k^{1/2},
\]
we have
\[
U(y) \ge U^{\bar \lda }(y) \quad \mbox{for } \bar \lda <|y|,
\]
which is a clear violation of \eqref{eq:f-1} because $\bar \lambda=2$.
This contradiction concludes the proof of Theorem \ref{sphe-har}.

\end{proof}

\begin{cor} \label{cor:sphere-harnack} Under the same assumptions in Theorem \ref{sphe-har}, we have
\[
\max_{r/2\le |x|\le 2r } u\le C_1 \min _{r/2\le |x|\le 2r } u
\]
for every $0<r<1/4$, where $C_1$ is independent of $r$. Moreover, for $0<|x|<1/4$,
\[
|\nabla u(x)| \le C_1|x|^{-1} u(x),
\]
\[
|\nabla^2 u(x)| \le C_1|x|^{-2} u(x).
\]

\end{cor}

\begin{proof} The corollary follows from Theorem \ref{sphe-har}  by using the standard local estimates for  the rescaled function $v(y)=r^{2}u(ry)$. We omit the details.

\end{proof}

\section{Lower bound and removability}
\label{s:3}

In this section, we shall show that either $0$ is a removable singularity or $u(x)|x|^{\frac{n-2}2}\ge c$ for some $c>0$ when $n=6$. This is based on a delicate analysis using  the Pohozaev identity.

We shall make a conformal change of the metric around the origin.
Suppose that  $\{y_1,\dots,y_n\}$ is a conformal normal coordinates system centered at $0$. Using the polar coordinates, we have
\[
g=\ud r^2 +r^2h(r,\theta),
\]
where $h$ is a metric on $\mathbb{S}^{n-1}$ and $\det h=1$, $r=|y|$ and $\theta= \frac{y}{|y|}$.
 Let
 \[
 f(r)= (1-r^2)^{-\frac{n-2}{2}},
 \] which is a solution of
\begin{equation} \label{eq:negative-case}
\Delta f=n(n-2) f^{\frac{n+2}{n-2}}.
\end{equation}  Let
\[
\tilde g= f^{\frac{4}{n-2}} g
\]
be a conformal metric of $g$, then the conformal covariance property of $L_g$ gives
\be
\begin{split}
c(n)R_{\tilde g}&=-L_{\tilde g}(1)= -f^{-\frac{n+2}{n-2}}L_{g} f\\&=  -f^{-\frac{n+2}{n-2}} (\Delta f +c(n)R_g f) = -n(n-2) +O(|y|^2),
\end{split}
\ee
where $|R_g|\le Cr^2$ in the conformal normal coordinates was used.
We shall use geodesic normal polar coordinates of $\tilde g$, in which
\be \label{eq:metric-conformal}
\tilde g= f^{\frac{4}{n-2}}\ud r^2 +f^{\frac{4}{n-2}} r^2h(r,\theta)= \ud \rho^2 +\rho^2 \tilde h(\rho, \theta),
\ee
where $\rho=\frac12 \ln \frac{1+r}{1-r}$ and
\[
\sqrt{\det \tilde h} = \sqrt{ \det \tilde g}= f^{\frac{2n}{n-2}}=(1-r^2)^{-n}=:\zeta(\rho).
\]
Then  the Laplace-Beltrami operator can be written as
\begin{align}
\Delta_{\tilde g} &=\pa_\rho^2+\frac{1}{\rho^{n-1} \zeta } \pa_\rho ( \rho^{n-1} \zeta  )\pa_\rho +\frac{1}{\rho^2} \Delta_{\tilde h} \nonumber \\&
=\pa_\rho^2 +\frac{n-1}{\rho }\pa_\rho +\pa_\rho \ln \zeta  \pa_\rho +\frac{1}{\rho^2} \Delta_{\tilde h}.
\label{eq:laplacian}
\end{align}
Suppose $u$ is a positive solution of
\begin{equation}\label{Y-10}
-L_{\tilde g} u=n(n-2) u^{\frac{n+2}{n-2}} \quad \mbox{in }B_1\setminus \{0\},
\end{equation}
 $\{x_1,\dots, x_n\}$ is a normal coordinates system of $\tilde g$ centered at $0$, we let
  \begin{equation*}
P(r, u):=\int_{\partial B_{r}}\left(\frac{n-2}{2} u \frac{\partial u}{\partial r}-\frac{1}{2} r|\nabla u|^{2}+r\left|\frac{\partial u}{\partial r}\right|^{2}+\frac{(n-2)^{2}}{2} r u^{\frac{2 n}{n-2}}\right) d S_{r}
\end{equation*}
be the Pohozaev integral, where $d S_{r}$ is the standard area measure on $\partial B_{r}$. The Pohozaev  identity asserts that, for any $0<s \leq r<1$,
\begin{equation} \label{eq:phozaev}
P(r, u)-P(s, u)=-\int_{s \leq|x| \leq r}\left(x^{k} \partial_{k} u+\frac{n-2}{2} u\right) (L_{\tilde g} u-\Delta u) \ud x.
\end{equation}
By Corollary \ref{cor:sphere-harnack}, we have
\[
\left|(x^{k} \partial_{k} u+\frac{n-2}{2} u) (L_{\tilde g} u-\Delta u) \right | \le C|x|^{2-n},
\]
which implies that the following limit can be defined:
\[
P(u):=\lim_{r\to 0} P(r,u).
\]

\begin{thm}\label{thm:remo} Assume $n=6$ and $u>0$ is a solution of \eqref{Y-10}. Then $P(u)\le 0$ and the equality holds if and only if $0$ is an removable singularity of $u$.
\end{thm}
 When $n=3,4,5$, Theorem \ref{thm:remo}  was proved by Marques \cite{marques-1} by an argument similar to that of Chen-Lin \cite{ChenLin3} for the prescribing scalar curvature.

\begin{proof}  If $0$ is removable, it is easy to check that $P(u)=0$.   Suppose $P(u)\ge 0$. We will show that $P(u)=0$ and $0$ is removable. Thus the theorem follows.

\medskip

\textbf{Claim 1.}
\be \label{eq:limitinf}
\liminf_{x\to 0} u(x)|x|^{\frac{n-2}{2}}=0.
\ee

With the establishment of the upper bound of $u$, the proof of (\ref{eq:limitinf}) is standard (See page 359 of \cite{marques-1}). Roughly speaking, if $u(x)|x|^{\frac{n-2}2}\ge c$, then for any $r_i\to 0$,
 $v_i(y)=r_i^{\frac{n-2}2}u(r_iy)$ converges along a subsequence to $v$ of
$$\Delta v+n(n-2)v^{\frac{n+2}{n-2}}=0\quad \mbox{in }\R^n\setminus \{0\}, $$
which has a non-removable singularity at the origin. By \cite{CGS}, $v$ is a Fowler solution and $P(v)<0$. Then we obtain a contradiction from
$$0>P(v)=P(1,v)=\lim_{i\to \infty}P(1, v_i)=\lim_{i\to \infty} P(r_i, u)=P(u)\ge 0. $$

\medskip

\textbf{Claim  2.}
\be \label{eq:limit}
\lim_{x\to 0} u(x)|x|^{\frac{n-2}{2}}=0.
\ee

In the geodesic normal poplar coordinates system, using \eqref{Y-10}, \eqref{eq:laplacian} and Corollary \ref{cor:sphere-harnack} we have
\begin{align}\label{main-eq-low}
\bar u_{\rho\rho }+\frac{n-1}{\rho} \bar u_\rho &=\dashint_{\pa B_\rho}\left(-\pa_\rho \ln \zeta  \pa_\rho u  -\frac{1}{\rho^2}\Delta_{\tilde h} u +c(n)R_{\tilde g} u - u^{\frac{n+2}{n-2}}\right),\nonumber\\&
=-\pa_\rho \ln \zeta   \bar u_\rho +\dashint_{\pa B_\rho}\left( c(n)R_{\tilde g} u - u^{\frac{n+2}{n-2}}\right)\\&
\le -\pa_\rho \ln \zeta   \bar u _\rho  - (n(n-2)+O(\rho^2)) \bar u  - c_2 \bar u ^{\frac{n+2}{n-2}}, \nonumber
\end{align}
where  $\bar u$ is the average of $u$ with the standard metric and $c_2>0$, and we have used that $\det \tilde h$ depends only on $\rho$ and
\[
\dashint_{\pa B_\rho} \Delta_{\tilde h} u \rho^{n-1}\,\ud vol_{g_{\mathbb{S}^{n-1}}}=  \frac{\rho^{n-1}}{\sqrt {\det \tilde h}} \dashint_{\pa B_\rho} \Delta_{\tilde h} u\,\ud vol_{\tilde h} =0.
\]
 Let $t=-\ln \rho$ and $ \bar u(\rho)=e^{\frac{n-2}{2} t} w(t) $. By a direct computation,
\[
\bar u_\rho= -e^{\frac{n}{2} t}\left(\frac{n-2}{2} w+ w_t\right),
\]
\[
\bar u_{\rho\rho}=e^{\frac{n+2}{2} t}\left(\frac{n(n-2)}{4} w+(n-1) w_t +w_{tt}\right).
\]
Therefore, we have
\[
\bar u_{\rho\rho }+\frac{n-1}{\rho} \bar u_\rho = e^{\frac{n+2}{2} t} \left(w_{tt}-(\frac{n-2}{2})^2 w\right),
\]
and
\begin{align*}
w_{tt}-(\frac{n-2}{2})^2 w \le e^{-t}[\pa_\rho \ln \zeta  (\frac{n-2}{2} w+ w_t )]-(n(n-2)+O(e^{-2t}))  e^{-2t} w-c_2 w^{\frac{n+2}{n-2}}.
\end{align*}
Note that
\[
r=\frac{e^{2\rho}-1}{e^{2\rho}+1}, \quad \frac{\ud r}{\ud \rho}= \frac{4 e^{2\rho}}{(e^{2\rho}+1)^2},  \quad  \frac{\ud^2 r}{\ud \rho^2} \Big|_{\rho =0}=0,
\]
and thus
\[
r= \rho +O(\rho^3).
\]
It follows that
\begin{align*}
\pa_\rho \ln \zeta&= \frac{2n r}{1-r^2} \frac{4 e^{2\rho}}{(e^{2\rho}+1)^2} =2n( \rho +O(\rho^3) ) (1+O(\rho^2)) (1+O(\rho ^2))\\&
=2n \rho +O(\rho^3)=2n e^{-t} +O(e^{-3t}) .
\end{align*}

Hence,
\begin{align}\label{help-1}
e^{-t}[\pa_\rho \ln \zeta  (\frac{n-2}{2} w+ w_t )]-[n(n-2)+O(e^{-2t})]  e^{-2t} w \nonumber\\
=e^{-2t} (2n +O(e^{-2t})) w_t +O(e^{-4t}) w.
\end{align}
 Thus the upper bound of $w_{tt}-(\frac{n-2}2)^2w$ can be determined as
\be \label{eq:main-obs}
w_{tt}-(\frac{n-2}{2})^2 w   \le e^{-2t} (2n +O(e^{-2t})) w_t +O(e^{-4t}) w -c_2 w^{\frac{n+2}{n-2}}.
\ee

By Corolllary \ref{cor:sphere-harnack}, we have  $|w_t(t)|\le C w(t)$.
Using (\ref{help-1})  and first two lines of (\ref{main-eq-low}), we obtain a lower bound of $w_{tt}-(\frac{n-2}2)^2w$:
\begin{equation}\label{eq:lower-inequality}
w_{tt}-(\frac{n-2}{2})^2 w \ge -c_1 w^{\frac{n+2}{n-2}}- c_3 e^{-2t} w.
\end{equation}

If Claim 2 were not true, by Claim  1 and Corollary \ref{cor:sphere-harnack}, we can choose $\va_0>0$  sufficiently small so that there exist sequences $\bar t_i \le t_i \le  t_i^*$ with $\lim_{i\to \infty} \bar t_i = +\infty$, such that
$w(\bar t_i) = w(t_i^*) = \va_0$, $w_t(t_i) = 0$, and $\lim_{i\to \infty} w(t_i) = 0$. Also the smallness of $w(t)$ implies 
\be \label{eq:convex}
\frac{1}{C} w\le w_{tt}\le C w  \quad \mbox{for }\bar t_i \le t\le t_i^*.
\ee
Hence, for $\bar t_i\le t\le t_i$, we have  $w_t\le 0$
\[
w_t(t) \le  -\frac{1}{C}\int_{t}^{t_i} w\,\ud s.
\]
It follows that for $\bar t_i\le t\le t_i-1$
\[
w_t(t) \le  -\frac{1}{C}\int_{t}^{t+1} w\,\ud s\le -\frac{1}{C} w(t+1) \le - \frac{1}{C} w(t),
\]
where we used Harnack inequality in Corolllary \ref{cor:sphere-harnack}. By \eqref{eq:main-obs}, we obtain, for large $i$
\be \label{eq:crucial}
w_{tt}-(\frac{n-2}{2})^2 w   \le -c_2 w^{\frac{n+2}{n-2}} \quad \mbox{for }\bar t_i\le t\le  t_i-1.
\ee
In conclusion,
\be \label{eq:24-0}
-c_1 w^{\frac{n+2}{n-2}}- c_3 e^{-2t} w \le w_{tt}-(\frac{n-2}{2})^2 w   \le c_3 e^{-4 t_i} w -c_2 w^{\frac{n+2}{n-2}} \quad \mbox{for }\bar t_i\le t\le  t_i.
\ee
and
\begin{equation} \label{eq:24}
-c_1 w^{\frac{n+2}{n-2}} -c_3 e^{-2t} w \le w_{tt}-\frac{n-2}{2} w \le -c_2 w^{\frac{n+2}{n-2}} +c_3 e^{-2t} w \quad \mbox{for }t_i\le t\le  t_i^* .
\end{equation}

Now we use (\ref{eq:24-0}) and (\ref{eq:24}) to derive pointwise estimates of $w(t)$. This part is similar to the proof of (27) and (28) in \cite{marques-1}, the main improvement is the first inequality of \eqref{eq:28}, where $e^{-4t_i}$ replaces $e^{-2\bar t_i}$ of (28) in \cite{marques-1}.

\begin{lem} \label{point-w} The following two estimates hold:
\begin{equation} \label{eq:27}
(\frac{2}{n-2} -ce^{-2t_i}) \ln \frac{w(t)}{w(t_i)} \le t-t_i \le (\frac{2}{n-2} +ce^{-2t_i}) \ln \frac{w(t)}{w(t_i)} +c
\end{equation}
for $t_i\le t\le t_i^*$, and
\begin{equation} \label{eq:28}
(\frac{2}{n-2} -ce^{-4 t_i}) \ln \frac{w(t)}{w(t_i)} \le t_i-t \le (\frac{2}{n-2} +ce^{-2\bar t_i}) \ln \frac{w(t)}{w(t_i)} +c
\end{equation}
for $\bar t_i \le t\le t_i$.
\end{lem}

\begin{proof}[Proof of Lemma \ref{point-w}] We only prove the first inequality in \eqref{eq:28}, since the other three were proved in \cite{marques-1}.  By the second inequality of (\ref{eq:24-0}) we have
$$w_{tt}-\Big((\frac{n-2}2)^2+c_3e^{-4t_i}\Big)w\le 0, \quad \bar t_i<t<t_i. $$
Multiplying $w'(t)$ (which is non-positive) on both sides we have
$$\frac{\ud}{\ud t}\left (w_t^2-((\frac{n-2}2)^2+c_3e^{-4t_i})w^2\right )\ge 0. $$
It follows that
\[
w_t(t)^2-\Big((\frac{n-2}2)^2+c_3e^{-4t_i}\Big)w(t)^2 \le -\Big((\frac{n-2}2)^2+c_3e^{-4t_i}\Big)w(t_i)^2 \quad \mbox{for }\bar t_i<t<t_i.
\]
Hence,
\[
\frac{\ud t}{\ud w}= \frac{1}{w_t}  \le -\Big((\frac{n-2}2)^2+c_3e^{-4t_i}\Big)^{-\frac{1}{2}} \frac{1}{\sqrt{w(t)^2- w(t_i)^2 }}.
\]
Integrating the above inequality, we have
\begin{align*}
t_i -t =- \int_{w(t_i)}^{w(t)} \frac{\ud t}{\ud w} \,\ud w &\ge (\frac{2}{n-2} -c e^{-4t_i}) \int_{w(t_i)}^{w(t)} \frac{1}{\sqrt{w^2- w(t_i)^2 }} \,\ud w\\&
\ge   (\frac{2}{n-2} -c e^{-4t_i}) \ln \frac{w(t)}{w(t_i)},
\end{align*}
where we have used the estimate
\begin{align*}
\int_{1}^{a} \frac{1}{\sqrt{s^2-1}}\,\ud s=  \int_0^{\ln a} \frac{e^{\xi}}{ \sqrt{e^{2\xi}-1}}\,\ud \xi \ge  \int_0^{\ln a} 1\,\ud \xi =\ln a \quad \mbox{for }a>1.
\end{align*}

Therefore, Lemma \ref{point-w} is proved.

\end{proof}

\medskip

At $|x|=\rho_i= e^{-t_i}$,  we have
\begin{equation} \label{eq:asy-sym}
u(x)=\bar u(r_i)(1+o(1)),\quad |\nabla u(x)|=-\bar u'(r_i)(1+o(1)).
\end{equation}
Indeed, let $h_i(y)=\frac{u(\rho_iy)}{u(\rho_i e_1)}$, where $e_1=(1,0,\dots, 0)$.  We have
$$L_{g_i}h_i+n(n-2) (\rho_i^{\frac{n-2}2}u_i(\rho_i e_1))^{\frac{4}{n-2}} h_i^{\frac{n+2}{n-2}}=0 \quad \mbox{in }B_{1/\rho_i}\setminus \{0\},$$
where $ (g_i)_{kl}=g_{kl}(\rho_iy)$.  By Corollary \ref{cor:sphere-harnack},  $h_i$ is locally uniformly bounded in $\R^n \setminus \{0\}$. By the choice of $\rho_i$, $\rho_i^{\frac{n-2}2}u_i(\rho_i e_1)\to 0$ as $i\to \infty$. Hence, $h_i \to h$ in $C^2_{loc}(\R^n \setminus \{0\}) $ for some $h$ satisfying
\[
-\Delta h=0 \quad \mbox{in }\R^n \setminus \{0\}, \quad h\ge 0
\]
and $h(e_1)=1$ and $\partial_\rho(h(y)\rho^{\frac{n-2}2})=0$. By the B\^ocher theorem,
$h(y)=a|y|^{2-n}+b$ with $a=b=\frac 12$.  Hence, \eqref{eq:asy-sym} follows.

By \eqref{eq:asy-sym}, we have
\begin{equation*}
P(\rho_{i}, u)=|\mathbb{S}^{n-1}|
\left(\frac{1}{2} w'\left(t_{i}\right)^{2}-\frac{1}{2}\left(\frac{n-2}{2}\right)^{2} w^{2}\left(t_{i}\right)+\frac{(n-2)^{2}}{2} w^{\frac{2 n}{n-2}}\left(t_{i}\right)\right)(1+o(1)).
\end{equation*}
Hence for sufficiently large $i$
\begin{equation}\label{temp-new-1}
w^{2}(t_{i}) \leq c_{n}|P(\rho_{i},u )|.
\end{equation}
By the choice of $t_i$, we have
\begin{equation}
\label{eq:poho-0}
P(u)=\lim _{i \to \infty} P(\rho_{i}, u)=0.
\end{equation}

It follows the Pohozaev identity \eqref{eq:phozaev} and \eqref{eq:poho-0} that
\begin{equation*}
\begin{split}
|P(\rho_i,u)|&\leq \int_{B_{\rho_{i}} \setminus  B_{\rho_{i}^{*}}}|\mathcal{A}(u)| \ud x +\int_{B_{\rho_{i}^{*}}}|\mathcal{A}(u)| \ud x \\&
=: I_1+I_2,
\end{split}
\end{equation*}
where $\rho_i^*=e^{-t_i^*}$,  $$
\mathcal{A}(u)=\left(x^{k} \partial_{k} u+\frac{n-2}{2} u\right) (L_{\tilde g}u-\Delta u).
$$
By Corollary \ref{cor:sphere-harnack}, we have
\[
|\mathcal{A}(u)| \le C |x|^{2-n}.
\]
Hence,
\[
I_2 \le C(\rho^*)^2 = C e^{-2t_i^*}.
\]
By the first inequality in \eqref{eq:27}, we have
\[
w(t)\le w(t_i) \exp \left( \Big(\frac{n-2}{2} +ce^{-2t_i }\Big)(t-t_i)\right),
\]
which implies
\[
u(x) \le c w(t_i) \exp \left(- \Big(\frac{n-2}{2} +ce^{-2t_i }\Big)t_i\right) |x|^{2-n-ce^{-2t_i}}\quad \mbox{for } \rho_i^*\le |x|\le \rho_i.
\]
By Corollary \ref{cor:sphere-harnack}, we have
\begin{align*}
|\mathcal{A}(u)|& \le C  u ^2.
\end{align*}
Hence,
\begin{align*}
I_1 \le C w(t_i)^2 e^{-(n-2) t_i}  \int_{\rho_i^*\le |x|\le \rho_i }  |x|^{4-2n-2ce^{-2t_i} }  \,\ud x  .
\end{align*}
By \eqref{eq:27} and \eqref{eq:28}, we see that
\[
t_i^*-t_i \le (\frac{2}{n-2}+ce^{-2t_i}) \ln \frac{\va_0}{ w(t_i)} +c, \quad t_i-\bar t_i \ge  (\frac{2}{n-2}-ce^{-4t_i}) \ln \frac{\va_0}{ w(t_i)}.
\]
Hence,
\be \label{eq:quotient}
\frac{t_i^*-t_i }{t_i-\bar t_i}\le 1+c e^{-2t_i}+C(\ln \frac{\va_0}{ w(t_i)})^{-1}.
\ee
Using the second inequality of \eqref{eq:28}, we have $(t_i-\bar t_i) (\ln \frac{\va_0}{ w(t_i)})^{-1} \le C$. Thus \eqref{eq:quotient} implies
\be \label{eq:t*}
t_i^*\le 2t_i -\bar t_i +C.
\ee
Using (\ref{eq:t*}) we can estimate $I_1$ more precisely:
\begin{align*}
I_1 &\le C w(t_i)^2 e^{-(n-2) t_i}  ( (\rho_i^*)^{4-n}- \rho_i ^{4-n})\\&
= C w(t_i)^2 e^{-(n-2) t_i} (e^{(n-4) t_i^*}-e^{(n-4) t_i} )\\&
\le C w(t_i)^2 (C e^{(n-6) t_i - (n-4) \bar t_i }-e^{-2 t_i} )\le C  w(t_i)^2 e^{-2\bar t_i},
\end{align*}
where in the final step we used $n=6$. Combing the estimates of $I_1$ and $I_2$, we have, for $n=6$,
\be \label{temp-new-2}
|P(\rho_i,u)|\leq C w(t_i)^2  e^{-2 \bar t_i}   +C  e^{-2t_i^*}.
\ee

Using \eqref{temp-new-1} and \eqref{temp-new-2}, we can combine terms to obtain
\be \label{eq:vanishing}
w(t_i)^2 \le  C e^{-2t_i^*}
\ee
for $i$ large.
From the first inequality of \eqref{eq:28} and the first inequality of \eqref{eq:27},  we have, for $n=6$,
\[
t_i-\bar t_i \ge (\frac12 -ce^{-4t_i}) \ln \frac{\va_0}{ w(t_i)}
\]
and
\[
t_i^*- t_i \ge (\frac12 -ce^{-2t_i}) \ln \frac{\va_0}{ w(t_i)}.
\]
Adding them up and using \eqref{eq:vanishing} and \eqref{eq:t*},  we have
\[
t_i^*- \bar t_i  \ge -(1 -ce^{-2t_i}) \ln  w(t_i) -C\ge (1 -ce^{-2t_i})   t_i^* -C \ge t_i^*-C,
\]
which implies
\[
\bar t_i \le C.
\]
This contradicts to $\bar t_i \to \infty$. Therefore, Claim 2 is proved.

Based on Claim 2 we clearly have $w'(t)<0$ for $t>T_1$.
Equation (\ref{eq:lower-inequality}) now implies
$$w_{tt}-(4-\delta)w\ge 0\quad \mbox{for } t\ge T_1, $$
where $\delta>0$ is some small constant. Thus for $t\ge T_1$,
$w_t^2-(4-\delta)w^2$ is non-increasing, and the integration of this quantity leads to
$$w(t)\le w(T_1)\exp(-(4-\delta)(t-T_1)),\quad t>T_1, $$
whose equivalent form is
$$u(x)\le C(\delta)|x|^{-\delta}. $$
Then standard elliptic estimate immediately implies that $u$ has a removable singularity at the origin.

Therefore, we complete the proof of Theorem \ref{thm:remo}.
\end{proof}

\begin{cor} \label{cor:lower-bound} Assume $n=6$ and $u>0$ is a solution of \eqref{Y-10}. If $0$ is not removable, then
\[
u(x)\ge \frac{1}{C}|x|^{-2},
\]
where $C>1$ is independent of $x$.
\end{cor}

\begin{proof} If it were false, then $\liminf_{x\to 0}|x| ^2 u(x)=0$.  As the proof of \eqref{eq:poho-0}, we have $P(u)=0$ and thus $0$ is removable. We obtain a contradiction. The corollary is proved.

\end{proof}

\begin{proof}[Proof of Theorem \ref{asy-s-1}]  Suppose that $0$ is not a removable singularity of $u$.
After conformal changes,  using Theorem \ref{sphe-har} and Corollary \ref{cor:lower-bound} we have
\[
\frac{1}{C}|x|^{-2} \le u(x)\le C|x|^{-2}.
\]
Theorem \ref{asy-s-1} follows immediately from Theorem 8 of \cite{marques-1}.

\end{proof}

\bigskip

\noindent J. Xiong

\noindent School of Mathematical Sciences, Beijing Normal University\\
Beijing 100875, China\\[1mm]
Email: \textsf{jx@bnu.edu.cn}

\medskip

\noindent L. Zhang

\noindent Department of Mathematics, University of Florida\\
        358 Little Hall P.O.Box 118105\\
        Gainesville FL 32611-8105, USA\\[1mm]
Email: \textsf{leizhang@ufl.edu}

\end{document}